\documentclass[11pt]{article}

\usepackage{dsfont}

\def\colorful{0}

\usepackage{amsthm,amsfonts,amsmath,amssymb,epsfig,color,float,graphicx,verbatim, enumitem}
\usepackage{multirow}
\usepackage{algorithm}
\usepackage[noend]{algpseudocode}

\newif\ifhyper\IfFileExists{hyperref.sty}{\hypertrue}{\hyperfalse}
\hypertrue
\ifhyper\usepackage{hyperref}\fi

\usepackage{enumitem}

\usepackage{framed}
\usepackage{nicefrac}

\def\nnewcolor{1}
\ifnum\nnewcolor=1

\fi
\ifnum\nnewcolor=0

\fi

\ifnum\colorful=1
\newcommand{\new}[1]{{\color{red} #1}}

\else
\newcommand{\new}[1]{{#1}}

\fi
\newcommand{\tr}{\mathrm{tr}}

\newtheorem{theorem}{Theorem}[section]

\newtheorem*{conj}{Conjecture}

\newtheorem{lemma}[theorem]{Lemma}
\newtheorem{informal theorem}[theorem]{Theorem (informal statement)}

\newtheorem{proposition}[theorem]{Proposition}

\newtheorem{claim}[theorem]{Claim}
\newtheorem{fact}[theorem]{Fact}

\theoremstyle{definition}

\newcommand{\R}{\mathbb{R}}

\newcommand{\E}{\mathbf{E}}
\newcommand{\eps}{\epsilon}

\renewcommand{\Pr}{\mathbf{Pr}}

\newcommand{\polylog}{\mathrm{polylog}}
\newcommand{\var}{\mathbf{Var}}

\newcommand{\calN}{{\cal N}}

\newcommand\blfootnotea[1]{%
  \begingroup
  \renewcommand\thefootnote{}\footnote{#1}%
  \endgroup
}

\title{A Nearly Tight Bound for Fitting an Ellipsoid \\
to Gaussian Random Points\blfootnotea{Author last names are in randomized order.}}

\author{
Daniel M. Kane\thanks{Supported by NSF Medium Award CCF-2107547, NSF Award CCF-1553288 (CAREER), and a grant from CasperLabs.}\\
University of California, San Diego\\
{\tt dakane@cs.ucsd.edu}
\and
Ilias Diakonikolas\thanks{Supported by NSF Medium Award CCF-2107079, NSF Award CCF-1652862 (CAREER), a Sloan Research Fellowship, and a DARPA Learning with Less Labels (LwLL) grant.}\\
University of Wisconsin-Madison\\
{\tt ilias@cs.wisc.edu}\\
}

\begin{document}

\maketitle

\begin{abstract}
We prove that for $c>0$ a sufficiently small universal constant that 
a random set of $c d^2/\log^4(d)$ independent Gaussian random points in $\R^d$ 
lie on a common ellipsoid with high probability. 
This nearly establishes a conjecture of~\cite{SaundersonCPW12}, within logarithmic factors.
The latter conjecture has attracted significant attention over the past decade, due
to its connections to machine learning and sum-of-squares lower bounds for certain statistical problems.
\end{abstract}

\section{Introduction} \label{sec:intro}

\subsection{Background} \label{ssec:background}

In this paper, we study the following question:
Given a multiset of $m$ i.i.d.\ samples from the standard multivariate Gaussian on $\R^d$,
does there exist an origin-centered ellipsoid that passes through all these points?
That is, we ask whether a set of $m$ {\em random points} has the so called {\em ellipsoid fitting property}.
Specifically, we are interested in understanding the largest value of $m$ such
that the ellipsoid fitting property holds with high probability. Throughout this paper,
the term ``with high probability'' will be used to mean $1-o_d(1)$, where $o_d(1)$ goes to $0$
when $d \rightarrow \infty$.

Fitting an ellipsoid to random points is a natural probabilistic question that is interesting on its own right.
The question, first studied in~\cite{SaundersonPhD, SaundersonCPW12}, has attracted significant
attention over the past years, in part due to its connections to theoretical computer science
and machine learning. The early papers~\cite{SaundersonPhD, SaundersonCPW12, SaundersonPW13}
formulated a plausible conjecture on the optimal value of the parameter $m$,
and made the first progress towards proving this conjecture. In more detail, these works conjectured
a phase transition at $m \approx d^2/4$. Formally, the conjecture posits the following:

\begin{conj}[SCPW conjecture] \label{conj}
For any universal constant $c>0$, (i) for $m \leq (1-c) d^2/4$ there exists an origin-centered fitting
ellipsoid with high probability, and (ii) for $m \geq (1+c) d^2/4$, there does not exist
an origin-centered fitting ellipsoid with high probability.
\end{conj}

As of now, the conjecture in its above precise form remains open.
Before we describe prior progress towards proving the SCPW conjecture, we briefly summarize
some known motivation from the perspectives of machine learning and theory of computation.
For a detailed description, the reader is referred to~\cite{PTVW22} (the first version of which
motivated the current work).

\paragraph{Motivation from ML and TCS}

The motivation for the introduction of the SCPW conjecture
in the initial works~\cite{SaundersonPhD, SaundersonCPW12, SaundersonPW13}
came from machine learning. Specifically, they studied the algorithmic problem
of decomposing a data matrix into a sum of a diagonal and a random low-rank matrix. It turns out
that the validity of the SCPW conjecture (or even a quantitatively weaker version thereof)
suffices for the correctness analysis of a convex programming formulation of this problem,
which was proposed in these works.
In a similar vein,~\cite{PodosinnikovaPW19} established
a connection \new{between} a closely related conjecture
to the problem overcomplete Independent Component Analysis (ICA).
Again, proving that modified conjecture would allow~\cite{PodosinnikovaPW19}
to analyze a specific SDP for the ICA problem. Finally,
there is a close connection between the positive side
of the SCPW Conjecture (i.e., part (i) above) and
lower bounds for Sums-of-Squares relaxations
for certain statistical (average case) tasks, specifically
in the context of the works~\cite{GhoshJJPR20, HsiehK22}.
Roughly speaking, proving the positive
side of the SCPW Conjecture (or a weaker version thereof)
suffices to prove that a certain SDP arising
(in these statistical problems) is feasible. Improved bounds
on the positive part of the SCPW conjecture
(like the one established in the current work)
have direct implications for these problems.

\paragraph{Prior Work Towards Proving the SCPW Conjecture}
The negative part of the SCPW conjecture (i.e., part (ii)) is easy to establish \new{to} within a factor of $2$.
As a result, most of the research has focuses on proving the positive part for as large value of $m$ as possible,
which we now describe. The early works that formulated the SCPW conjecture~\cite{SaundersonPhD, SaundersonPW13}
showed a bound of $m = O(d^{6/5-c})$, for any constant $c>0$. In the recent work~\cite{GhoshJJPR20},
this bound was improved to $m = O(d^{3/2-c})$, for any constant $c>0$. Interestingly, the latter improvement
was obtained indirectly, in the context of proving SoS lower bounds for the Sherrington-Kirkpatrick Hamiltonian.
More recently, the first version of~\cite{PTVW22} claimed a bound of $m = d^2 / \polylog(d)$. Unfortunately,
there were issues with their proof 
and the first version of the paper was retracted. 
The best correct bound, prior to our work, 
was the one obtained in~\cite{GhoshJJPR20}.

\subsection{Our Result and Techniques} \label{ssec:results}

\paragraph{Main Result} \new{We establish a bound on the positive part of the SCPW conjecture
that is near-optimal, within polylogarithmic factors.}

\begin{theorem}[Main Result]\label{thm:main}
There exists a universal constant $c>0$ such that the following holds.
For any positive integer $d$ and any $m < c d^2/\log^4(d)$,
if $x^{(1)},x^{(2)},\ldots,x^{(m)}$ are i.i.d.\ samples from $\calN(0,I_d)$,
then with probability at least $1-o_d(1)$ there exists
an origin-centered ellipsoid that passes through all of the $x^{(i)}$'s.
\end{theorem}

\paragraph{Overview of Techniques}

Our basic technique will be to use a so-called identity perturbation construction.
In particular, if $x^{(1)},x^{(2)},\ldots,x^{(m)}$ are i.i.d.\ samples from
$\calN(0,(1/d) I_d)$,
then it is our goal to show that with high probability there exists
a positive definite matrix $N$ such that $(x^{(i)})^\top N x^{(i)} = 1$, for all $1\leq i \leq m$.
As a first observation, we note that (with high probability) taking $N=I_d$
should \emph{nearly} work for all $i$.
We will thus attempt to construct an appropriate positive definite matrix 
$N$ that is close to the identity.

In particular, we will take $N = I+ \sum_{i=1}^m \delta_i x^{(i)}(x^{(i)})^\top$,
for some carefully chosen scalars $\delta_1,\delta_2,\ldots,\delta_m$.
In particular, since we have $m$ different $\delta_i$'s to set
in order to satisfy the $m$ linear (in the $\delta_i$) equations $(x^{(i)})^\top N x^{(i)} = 1$,
we can generically expect for there to be a unique solution.
This leaves us with the issue of our other constraint, namely that $N$ is positive definite.
For this it will help to know that $N$ is close to the identity,
as it will be sufficient to show that
$\left\|\sum_{i=1}^m \delta_i x^{(i)}(x^{(i)})^\top\right\|_2 \leq 1$.

To achieve this, we will need to first get a reasonable handle on the scalars $\delta_i$.
To that end, we derive a formula that $\delta = M^{-1}\eps$,
where $M$ is the matrix with entries $M_{i,j}$ roughly $(x^{(i)} \cdot x^{(j)})^2$
and $\eps_i$ a function of $\|x^{(i)}\|_2^2 -1$. The first thing to note about $M$
is that its diagonal entries will be approximately $1$, and its off-diagonal entries approximately $1/d$.
Thus, $M$ can be written as $(1/d) \mathbf{1}\mathbf{1}^\top + (1-1/d)I + A$,
for some matrix $A$ whose entries can be thought of as relatively small random noise.
Our first major task will be to show that $A$ has bounded operator norm.
In particular, we show that for $m\ll d^2/\log^4(d)$ we have that $\|A\|_2 < 1/2.$
The proof of this statement relies on the method of moments, namely bounding $\E[A^t]$
for some appropriate even integer $t$. This bound is particularly important,
as it allows us to bound $\|M^{-1}\|_2$, and in particular implies that $\delta$
is not much bigger than $\eps$.

This leaves us with the task of bounding the operator norm of
$R:= \sum_{i=1}^m \delta_i x^{(i)}(x^{(i)})^\top$.
We will establish this using a cover argument.
In particular, it suffices to show (see Fact \ref{cover lemma})
 that $|u^\top R u|$ is small for all $u$ in a cover of the unit sphere of $\R^d$.
This in turn means that it suffices to show that, for any particular unit vector $u$,
$|u^\top R u|$ is small with exponentially high probability.
To show this, we note that $u^\top R u$ is linear in $\delta$,
and thus linear in $\eps$. We would thus like to be able to treat this
as a sum of independent random variables
and apply some concentration bounds. Unfortunately,
as we have set things up, $\eps$ and $M$ are not independent of each other.
To address this issue, we need a slight change in definitions.
In particular, we let $x^{(i)} = (1+\eps_i)^{-1/2} v^{(i)}$ for $v^{(i)}$ the unit vector
in the direction of $x^{(i)}$ and $\eps_i = 1/\|x^{(i)}\|_2^2 -1$.
Now if we instead write $N = I+ \sum_{i=1}^m \delta_i v^{(i)}(v^{(i)})^\top$,
we have that $\delta = M^{-1} \eps$, where $M_{i,j} = (v^{(i)}\cdot v^{(j)})^2$.
Since the $v^{(i)}$ and $\eps_i$ are all independent of each other,
$\eps_i$ is independent of $M$ and $u$ and this strategy can be made to work.

However, to implement the above approach, 
we run into the additional technical issue that $|\eps_i|$ is unbounded.
This can be fixed by showing that, with high probability, $|\eps_i|$
is reasonably small for all $i$. Then conditioning on this event,
the $\eps_i$ are independent and bounded random variables.

For the rest of the cover argument, we have that $u^\top R u = \eps \cdot (M^{-1} \beta)$,
where $\beta_i = (u\cdot v^{(i)})^2$. If we have an $\ell_2$ bound on $\beta$
(and thus $M^{-1} \beta$), we can apply a Hoeffding bound to get
a high probability upper bound on $|u^\top R u|$. Unfortunately,
$\|\beta\|_2$ is not quite small enough for our purposes with exponentially high probability.
This is because there is an inverse exponential probability
that any given entry of $\beta$ will be constant sized.
To deal with this, we need to show that with exponentially high probability
$\beta$ can be divided into a heavy part --- 
consisting of a small number large coordinates, 
where we bound the contribution of these 
based on a high-probability $\ell_\infty$ bound on $\delta$
--- and a light part with $\ell_2$-norm at most $O(m/d^2)$.

\paragraph{Independent Work}
Independent work by~\cite{PTVW22}, using different techniques,
established a qualitatively similar bound of $m = d^2 / \polylog(d)$
on the positive part of the SCPW conjecture. Roughly speaking, the proof given in
\cite{PTVW22} relies on the least squares construction, 
taking $N$ to be the matrix with $(x^{(i)})^\top N x^{(i)}=1$, for all $i$, that minimizes $\|N\|_F$.
In contrast, our argument uses the identity perturbation construction, 
whereby $N$ is taken to be a perturbation of $I$ obtained by solving a system of linear equations.

\section{Preliminaries} \label{sec:prelims}

{\bf Notation. } For a vector $u \in \R^d$, we use $\|u\|_2$ to denote its $\ell_2$-norm.
For vectors $u, v \in \R^d$, we denote their standard inner product by $u \cdot v$.
We will sometimes use $(u)_i$ for the $i^{\mathrm{th}}$ coordinate of vector $u$.
We will use $I_d$ for the $d \times d$ identity matrix. If the dimension $d$ is clear
from the context, we will omit the index and simply use $I$.
For a square matrix $A$, we will use $\|A\|_2$ for its operator (spectral) norm.

We use $o_d(1)$ to denote a quantity that goes to $0$ as $d$ goes to infinity.
Throughout this paper, we use the phrase ``with high probability'' 
to mean with probability $1-o_d(1),$ and the phrase ``with exponentially high probability'' 
to mean with probability at least $1-2^{-Cd}$, for $C$ some sufficiently large constant.

We will make use of some basic facts about the uniform distribution over the 
sphere in $d$-dimensions. In particular, we will use the following basic lemma.

\begin{lemma}\label{sphere facts lemma}
If $v$ is a uniform random point on the unit sphere 
and $w$ another unit vector then the following holds: 
$\E[v\cdot w] = 0$, $\E[(v\cdot w)^2]=1/d$, 
and for any constant $k>0$, $\E[|v\cdot w|^k] = O_k(d^{-k/2})$. Furthermore,
$\Pr(|v\cdot w| > t) = \exp(-\Omega(dt^2))$.
\end{lemma}
\begin{proof}
Most of these statements follow from the fact that $x \sim \calN(0,(1/d) I_d)$ can be written 
as $L v$ for $L$ a scalar-valued random variable independent of $v$ with $L\in [1/2,2]$ with high probability. 
The second moment bound follows by extending $w$ to an orthonormal basis $w_1,w_2,\ldots,w_d$ of $\R^d$ 
and noting that $1 = \E\left[ \sum_{i=1}^d (v\cdot w_i)^2 \right] = d \sum_{i=1}^d \E[(v\cdot w)^2]$
by symmetry.
\end{proof}

We also note that Theorem \ref{thm:main} is trivial for any constant value of $d$, 
and thus we assume throughout that $d$ is at least a sufficiently large constant. 
We will also assume as necessary that $m = \Omega(d^2/\log^4(d))$ 
as increasing the value of $m$ only makes Theorem \ref{thm:main} harder.

\section{Proof of Main Result}

The basic outline of our proof is as follows:
In Section~\ref{structure sec}, we lay out the overall structure of the proof along
with the basic definitions of $M,N,\eps$, and $\delta$.
In Section \ref{norm bound sec}, we prove the bound on $\|A\|_2$.
In Section \ref{bounds sec}, we use this to prove some basic bounds on $M, \eps$ and $\delta$.
Finally, in Section \ref{cover sec}, we complete the details of the cover argument.

\subsection{Proof Structure}\label{structure sec}

To prove the theorem, we need to exhibit a positive definite matrix $N$
such that, with high probability over the $x^{(i)}$'s, \new{where $x^{(i)} \sim \calN(0,(1/d) I_d)$,}
we have that
$(x^{(i)})^{\top} N x^{(i)} = 1$ for all $1\leq i \leq m$.
As each of the terms $(x^{(i)})^{\top} N x^{(i)}$ will be close to $1$
with high probability, our goal will be to let $N$ be a perturbation of the identity matrix.
In particular, our plan will be to let
$N = I + \sum_{i=1}^m \delta_i x^{(i)} (x^{(i)})^{\top}$
for carefully chosen scalars $\delta_i$.
In particular, the constraints that $(x^{(i)})^{\top} N x^{(i)} = 1$ for each $i$
will give us a system of linear equations that hopefully uniquely define the $\delta_i$'s.

For parts of this argument that will become relevant later,
it will be useful for us to separate out the
direction of each vector $x^{(i)}$ from its length.
To this end, we express each $x^{(i)}$ as $x^{(i)} = (1+\eps_i)^{-1/2} v^{(i)}$,
where $v^{(i)}$ is a unit vector and $\eps_i \in (-1,\infty)$ is a real number
equal to $\eps_i := 1/\|x^{(i)}\|_2^2 - 1$.
We note that this specific parameterization of the length
(rather than selecting, e.g., $x^{(i)}$ equals $(1+\eps_i)v^{(i)}$ or $\ell_i v^{(i)}$) is
such that the resulting formula for the matrix $N$ below will be convenient for our purposes.
In particular, if we let $N:= I_d + \sum_{i=1}^m \delta_i v^{(i)} (v^{(i)})^{\top}$, we will have that
\begin{align}
(x^{(j)})^{\top} N x^{(j)} & = (x^{(j)})^{\top} I x^{(j)} + \sum_{i=1}^m \delta_i (x^{(j)} \cdot v^{(i)})^2 \nonumber \\
&= (1+\eps_j)^{-1} \left(1 + \sum_{i=1}^m (v^{(i)}\cdot v^{(j)})^2 \delta_i \right) \;. \label{eqn:ena}
\end{align}
We will henceforth use the notation
$\delta := (\delta_1, \ldots, \delta_m)$ and $\eps := (\eps_1, \ldots, \eps_m)$.
By \eqref{eqn:ena}, we need a choice of $\delta$ such that
$\sum_{i=1}^m (v^{(j)}\cdot v^{(i)})^2 \delta_i = \eps_j$ for all $j$.
Letting $M$ be the $m\times m$ matrix \new{given by }
\begin{equation}
M_{i,j} := (v^{(i)}\cdot v^{(j)})^2 \;, \label{eqn:M}
\end{equation}
\new{the relationship we need to hold between $\delta$ and $\eps$} is equivalent to having
that $\delta = M^{-1}\eps$.
Hence, so long as the matrix $M$ is invertible,
taking this choice of $\delta$ will give us a matrix $N$
such that $(x^{(i)})^{\top} N x^{(i)} = 1$ for all $i$.

It remains to show that the matrix $N$ is positive definite.
We will establish this by showing that the matrix
\begin{equation}
K:= N-I = \sum_{i=1}^m \delta_i v_i v_i^{\top} \label{eqn:K}
\end{equation}
satisfies $\|K\|_2 \leq 1$.
Towards this end, we will use an appropriate cover argument.
In particular, we will require the following standard fact.
\begin{fact}[see, e.g.,~\cite{Versh18}]\label{cover lemma}
There exists a set $\mathcal{C} \subseteq \R^d$ of $2^{O(d)}$
unit vectors such that for any symmetric matrix $R$ we have
$$
\| R \|_2 \leq 2 \max_{u\in\mathcal{C}} |u^{\top} R u| \;.
$$
\end{fact}
\noindent Thus, it suffices for us to show that with high probability
over the choice of $x^{(i)}$'s, that the following holds:
\begin{enumerate}
\item $M$ is invertible.
\item For each $u\in\mathcal{C}$, we have that
$\left| \sum_{i=1}^m \delta_i (u\cdot v_i)^2 \right| \leq 1/2$.
\end{enumerate}
Our general strategy for proving the second result above
will be a union bound over $u$'s. However, our analysis
will also need to depend on several events
that are independent of the choice of $u$.
Thus, we will find some collection of events
that imply the above conditions so that each event
satisfies the following:
\begin{itemize}
\item Either the event is independent of $u$ and occurs with probability $1-o_d(1)$.
\item Or the event depends on $u$ and occurs with probability
at least $1-2^{-Cd}$ for some sufficiently large constant $C>0$ assuming that all events of the first type hold.
\end{itemize}
By a union bound, all events of the first type hold with probability $1-o_d(1)$
and all events of the second type hold for all $u\in \mathcal{C}$,
again with probability $1-o_d(1)$. Thus, together these will imply
that the $x_i$'s lie on an origin-centered ellipsoid
with probability at least $1-o_d(1)$.

Perhaps the most significant of these events has to do
with the behavior of the matrix $M$. First, it is clear
that $\E[(v^{(i)}\cdot v^{(j)})^2]$ is equal to $1$ if $i=j$
and equal to $1/d$ if $i\neq j$.
Therefore, we have that
$\E[M] = (1-1/d) I + (1/d) \mathbf{1}\mathbf{1}^{\top}$,
where $\mathbf{1} = (1,1,\ldots,1)$ is the vector
with all entries equal to $1$.
Let $A$ be the difference, in particular, let
$$
A = M - (1-1/d)I - (1/d) \mathbf{1}\mathbf{1}^{\top} \;.
$$
Our most important claim is that with high probability $A$
has small operator norm.
Specifically, we prove the following proposition.

\begin{proposition}\label{A norm prop}
With probability $1-o_d(1)$ we have that $\|A\|_2 < 1/2$.
\end{proposition}

Note that this condition depends only on the $v^{(i)}$'s
(and in particular is independent of the $\eps_i$'s).
We prove it in the next section.

\subsection{Operator Norm Bounds: Proof of Proposition~\ref{A norm prop}}\label{norm bound sec}

The main goal of this section will be to prove Proposition \ref{A norm prop}.
We begin by finding a new useful way to express the matrix $A$.
\new{From the definition of $A$, it is easy to see} that $A_{i,j}$ is equal to $0$ if $i=j$ and
equal to $(v^{(i)}\cdot v^{(j)})^2 -1/d$ if $i\neq j$.
\new{We begin our analysis by introducing some new terminology that allows us to put this in a simpler form.}
In particular, let $S$ be the set of degree-$2$ polynomials on $\R^d$
given by $p(y) = y_i y_j$ for some $1\leq i < j \leq d$
or $p(y) = (y_i^2-1/d)/\sqrt{2}$ for some $1\leq i\leq d.$
We note that these are the Hermite polynomials of degree-$2$ on $\R^d$.
We have the following:

\begin{claim} \label{clm:A-rewrite}
For $i\neq j$, it holds that $A_{i,j} = 2 \sum_{p\in S} p(v^{(i)})p(v^{(j)})$.
\end{claim}
\begin{proof}
For $i\neq j$, we have the following chain of equalities:
\begin{align*}
A_{i,j} & = (v^{(i)}\cdot v^{(j)})^2-1/d\\
& = \left(\sum_{k=1}^d v^{(i)}_k v^{(j)}_k\right)^2 -1/d\\
& = \sum_{k,\ell=1}^d v^{(i)}_k v^{(i)}_\ell v^{(j)}_k v^{(j)}_\ell -1/d\\
& = 2 \sum_{1\leq k < \ell \leq d} v^{(i)}_kv^{(i)}_\ell v^{(j)}_k v^{(j)}_\ell + \sum_{1\leq k \leq d} (v^{(i)}_k)^2 (v^{(j)}_k)^2 -1/d\\
& = 2 \sum_{p(y)=y_ky_\ell, 1\leq k < \ell \leq d}p(v^{(i)})p(v^{(j)}) + \sum_{1\leq k \leq d} (v^{(i)}_k)^2 (v^{(j)}_k)^2
- (1/d)\sum_{1\leq k \leq d} (v^{(i)}_k)^2 - (1/d) \sum_{1\leq k \leq d} (v^{(j)}_k)^2 + 1/d\\
& = 2 \sum_{p(y)=y_ky_\ell, 1\leq k < \ell \leq d}p(v^{(i)})p(v^{(j)})
+ \sum_{1\leq k \leq d} \left( (v^{(i)}_k)^2-1/d \right) \left( (v^{(j)}_k)^2-1/d \right)\\
& = 2 \sum_{p\in S} p(v^{(i)})p(v^{(j)}) \;.
\end{align*}
This completes the proof of the claim.
\end{proof}
To proceed with our analysis,
we will split $S$ into two sets.
Namely, $S'$ is the set of polynomials of the form
$p(y) = y_iy_j$ for $1\leq i < j\leq d$,
and $S^\ast$ the set of polynomials of the form $(y_i^2-1/d)/\sqrt{2}$.
We can then write $A=A' + A^\ast$, where for $i\neq j$
we have that
$$
A_{i,j}' := 2\sum_{p\in S'} p(v^{(i)})p(v^{(j)}) \;,
$$
and
$$
A_{i,j}^\ast := 2\sum_{p\in S^\ast} p(v^{(i)})p(v^{(j)}) \;.
$$
\new{In order to bound the operator norm of $A$,}
we will bound the operator norms
of the matrices $A'$ and $A^{\ast}$ separately.

The case of $A^\ast$ turns out to be relatively simple.
We show:

\begin{lemma} \label{lem:A-star}
With high probability, we have that $\|A^{\ast}\|_2 \leq 1/4$.
\end{lemma}

\begin{proof}
Let $B$ be the $m\times d$ matrix defined as $B_{i,k} = p_k(v^{(i)})$,
where $p_k\in S^\ast$ is given by $p_k(y) = (y_k^2-1/d)/\sqrt{2}$.
We note that $A^\ast-2BB^{\top}$ has $0$ as its off-diagonal entries
and has diagonal entries equal to $-\sum_{k=1}^d ((v^{(i)}_k)^2-1/d)^2$.

Recalling that each $x^{(i)} \sim N(0, I_d/d)$,
by its definition each $v^{(i)}$ is a uniform random point in the unit sphere.
Therefore,  with high probability
we have that $|v^{(i)}_k| \leq (\log(d)/\sqrt{d})$
for all $1\leq i\leq m$ and $1\leq k \leq d$.
If this event holds, then the entries of $A^\ast-2BB^{\top}$
are all $O(\log^4(d)/d)$ in magnitude,
and thus for $d$ sufficiently large,
with high probability we have that
$\|A^\ast - 2 B B^{\top}\|_2 \leq 1/12$.
In order to bound $\|A^\ast\|_2$ from above,
we thus need \new{to prove an upper bound on}
$\|2 B B^{\top}\|_2 = 2 \|B^{\top} B\|_2$.

To achieve this, we note that the $(k,\ell)$-entry of $2B^{\top}B$ is
$$
\sum_{i=1}^m ((v^{(i)}_k)^2-1/d)((v^{(i)}_\ell)^2-1/d) \;.
$$
We first deal separately with the diagonal entries.
If the high probability condition that each $|v^{(i)}_k| < \log(d)/\sqrt{d}$ holds,
then each diagonal entry is at most $m \log^2(d)/d^2 < 1/12$.
Thus, letting $C$ be the matrix obtained by zeroing
out the diagonal entries of $2B^{\top}B$,
we have that with high probability
$\|2B^{\top}B-C\|_2 < 1/12.$ On the other hand,
we have that
$\|C\|_2 \leq \|C\|_F$.
The expected squared Frobenius norm of $C$
can be written as:
\begin{align*}
\E[\|C\|_F^2] & = \sum_{k\neq \ell} \E\left[ \left(\sum_{i=1}^m ((v^{(i)}_k)^2-1/d)((v^{(i)}_\ell)^2-1/d)\right)^2\right] \;.
\end{align*}
Note that for any fixed $k$ and $\ell$
the $((v^{(i)}_k)^2-1/d)((v^{(i)}_\ell)^2-1/d)$ are i.i.d.\ random variables,
and thus the RHS above can be simplified as follows:
$$
\E\left[ \left(\sum_{i=1}^m ((v^{(i)}_k)^2-1/d)((v^{(i)}_\ell)^2-1/d)\right)^2\right]
= m^2 \E[(w_k^2-1/d)(w_\ell^2-1/d)]^2 + m\var[(w_k^2-1/d)(w_\ell^2-1/d)] \;,
$$
where $w$ is a uniform random point on the unit sphere.
To analyze this quantity, we can write:
\begin{align*}
0 & = \E\left[\left( \sum_{k=1}^d (w_k^2 - 1/d) \right)^2\right]\\
&= d(d-1) \E[(w_k^2-1/d)(w_\ell^2-1/d)] + d\E[(w_k^2-1/d)^2]\\
&= d(d-1)\E[(w_k^2-1/d)(w_\ell^2-1/d)] + O(1/d) \;.
\end{align*}
Thus, we have that
$\E[(w_k^2-1/d)(w_\ell^2-1/d)] = O(1/d^3).$
It \new{follows from Lemma \ref{sphere facts lemma} that}\\ \mbox{$\var[(w_k^2-1/d)(w_\ell^2-1/d)] = O(1/d^4)$.}
Combining with the above, we have that
$$
\E[\|C\|_F^2] = O(d^2(m^2/d^6 + m/d^4)) = O((m/d^2)^2 + (m/d^2)) = o_d(1) \;,
$$
where the last equation follows by our choice of $m$.
Therefore, by the Markov inequality,
with probability $1-o_d(1)$ we have that
$\|C\|_2 \leq \|C\|_F \leq 1/12.$

In summary, with high probability,
we have that
$$
\|A^\ast\|_2 \leq \|A^\ast -2BB^{\top}\|_2 + \|2B^{\top}B-C\|_2 +\|C\|_2 \leq 1/12+1/12+1/12 = 1/4 \;.
$$
This completes the proof of Lemma~\ref{lem:A-star}.
\end{proof}

\noindent We next bound above $\|A'\|_2$.

\begin{lemma} \label{lem:A-prime}
With probability at least $1-1/d$, we have that $\|A'\|_2 \leq 1/4$.
\end{lemma}

As Lemma \ref{lem:A-prime} is substantially the most technically
difficult part of our proof, and the only part where the requirement that $m\ll d^2/\log^4(d)$ is required,
we start by providing a few words on our overall approach.
Our high-level strategy is to use the method of moments.
In particular, we will bound $\E[\tr((A')^t)]$ for some even integer $t$
on the order of $\log(d)$, as this will allow us to prove high probability bounds
on $\|A'\|_2$. We do this by expanding the $\tr((A')^t)$ and noting that
the vast majority of the remaining terms have mean $0$
and the ones that remain have easily-computable expectations.
We then use combinatorial techniques to bound the number of remaining terms and get our final result.

\begin{proof}
Note that $A_{i,j}$ is equal to $0$ if $i\neq j$,
and otherwise is equal to $\sum_{1\leq a\neq b\leq d} v^{(i)}_av^{(i)}_bv^{(j)}_av^{(j)}_b$.
Therefore, taking $t$ to be a positive even integer, we have that
$$
\tr((A')^t) = \sum_{\substack{i_1,i_2,\ldots,i_t \in\{1,2,\ldots,m\}\\ i_s \neq i_{s+1}}} \sum_{\substack{a_1,b_1,a_2,b_2,\ldots,a_t,b_t\\ a_s\neq b_s}} \prod_{s=1}^t v^{(i_{s})}_{a_s}v^{(i_{s+1})}_{a_s}v^{(i_s)}_{b_s}v^{(i_{s+1})}_{b_s} \;,
$$
where we use the convention that the indices are taken modulo $t$,
namely that $i_{t+1}=i_1, a_{t+1}=a_1$, and $b_{t+1}=b_1$.
Now we would like to bound the expectation of this quantity
over the choice of $v^{(1)},\ldots,v^{(m)}$ independent random unit vectors.

First, we claim that for any monomial $p(v^{(1)},v^{(2)},\ldots,v^{(m)})$ we have
that 
$$\E[p(v^{(1)},v^{(2)},\ldots,v^{(m)})] \leq \E[p(x^{(1)},x^{(2)},\ldots,x^{(m)})] \;,$$ 
\new{where the $x^{(i)}$ are independent $\calN(0, (1/d) I_d)$ random vectors}.
This holds because the distributions on both the $v$'s and the $x$'s
are symmetric in each coordinate, and thus both expectations are $0$
unless $p$ is even degree in each coordinate of each $v^{(i)}$ or $x^{(i)}$. We next recall that $x^{(i)} = L_i v^{(i)}$,
for independent random variables $L_i$ with $\E[L_i^2]=1$.
Since $p$ is a monomial, we have that
$p(x^{(1)},\ldots,x^{(m)}) = p(v^{(1)},\ldots,v^{(m)})L_1^{d_1}L_2^{d_2}\cdots L_m^{d_m}$,
where $d_i$ is the degree of $x^{(i)}$ in $p$ is a non-negative even number.
Since $\E[L_i^2]=1$, it follows that $\E[L_i^{d_i}] \geq 1$.
Thus, we get that $\E[p(v^{(1)},v^{(2)},\ldots,v^{(m)})] \leq \E[p(x^{(1)},x^{(2)},\ldots,x^{(m)})]$,
as desired.

Hence, we have that
\begin{equation} \label{eqn:trace-t}
\E[\tr((A')^t)] \leq \sum_{\substack{i_1,i_2,\ldots,i_t \in\{1,2,\ldots,m\}\\ i_s \neq i_{s+1}}} \sum_{\substack{a_1,b_1,a_2,b_2,\ldots,a_t,b_t\\ a_s\neq b_s}}
\E\left[\prod_{s=1}^t x^{(i_{s})}_{a_s}x^{(i_{s+1})}_{a_s}x^{(i_{s})}_{b_s}x^{(i_{s+1})}_{b_s}\right] \;.
\end{equation}
Let us consider the expectation inside the sum.
This is a monomial of a Gaussian.
Note that if $y$ is the concatenation of the $x^{(i)}$'s,
it is distributed as $\calN(0, (1/d) I_{md})$, and we have that
$$
\E\left[ \prod_{i=1}^{md} y_i^{d_i} \right]
$$
is $0$ unless all of the $d_i$'s are even,
and otherwise is equal to $d^{-\sum_{i=1}^{md} d_i/2} \prod_{i=1}^{md} (d_i-1)!!$.
Note that if $D=\sum_{i=1}^{md} d_i$ and we rewrite the above monomial as
$$
\E\left[\prod_{j=1}^D y_{c_j} \right] \;,
$$
where $i$ occurs as a value of $c_j$ exactly $d_i$ times,
then the expectation is equal to
$d^{-D/2}$ times the number of pairings of $\{1,2,\ldots,D\}$
such that $j$ is paired with $j'$ only if $c_j=c_{j'}$.
This is because the number of ways to pair the $j$'s
so that $c_j=i$ is $(d_i-1)!!$, if $d_i$ is even, and $0$ otherwise.

Therefore, the inner expectation in \eqref{eqn:trace-t},
$$
\E\left[\prod_{s=1}^t x^{(i_{s})}_{a_s}x^{(i_{s+1})}_{a_s}x^{(i_{s})}_{b_s}x^{(i_{s+1})}_{b_s}\right],
$$
equals $d^{-2t}$ times the number of pairings on the $4t$ formal symbols
$x^{(i_{s})}_{a_s}, x^{(i_{s+1})}_{a_s}, x^{(i_{s})}_{b_s}, x^{(i_{s+1})}_{b_s}$,
for $1\leq s\leq t$, such that if $x^{(i_{s})}_{\alpha_{s'}}$ is paired
with $x^{(i_{r})}_{\beta_{r'}}$ then $i_s = i_r$, and $\alpha_{s'} = \beta_{r'}$.

Note that any such pairing on these formal symbols
will partition the $i_s$-terms 
into some $\alpha$ many equivalence classes \new{for some $\alpha$},
and the $a_s, b_s$ terms into some $\beta$ \new{many} equivalence classes.
We note that if $i_s$ is equivalent to $i_{s+1}$ or $a_s$ is equivalent to $b_s$ for some $s$,
then this pairing will never show up
for any setting of the $i$'s, $a$'s, and $b$'s.
Otherwise, there will be at most $m^\alpha d^\beta$ settings of these variables consistent with this pairing. Call a pairing for which no $i_s$ is equivalent to $i_{s+1}$ or $a_s$ equivalent to $b_s$ \emph{valid}.

We have that
\begin{align*}
\E[\tr((A')^t)] & \leq \sum_{\substack{i_1,i_2,\ldots,i_t \in\{1,2,\ldots,m\}\\ i_s \neq i_{s+1}}} \sum_{\substack{a_1,b_1,a_2,b_2,\ldots,a_t,b_t\\ a_s\neq b_s}} d^{-2t}\\
& \ \ \ \ \ \ \ \ \ \ \ \ \ \ \ \ \ \ \ \ \ \ \  \cdot[\textrm{Number of valid pairings consistent with the choice of }i \textrm{'s}, a \textrm{'s}, b \textrm{'s}]\\
& \leq d^{-2t} \sum_{\textrm{valid pairings}} m^\alpha d^\beta.
\end{align*}
To bound this sum, we need the following claim:

\begin{claim} \label{clm:ind}
For $t$ any positive integer and for any valid pairing, we have that
$m^\alpha d^\beta \leq m^{t/2+1}d^t$.
\end{claim}

\begin{proof}
We prove this by induction on $t$.
As a base case we consider $t\leq 2$.
For $t=1$, there is no valid pairing since $i_1=i_2$.
For $t=2$, we have that $\alpha\leq 2$.
Since $x^{(i_{s})}_{a_s}$ cannot be paired with
$x^{(i_{s+1})}_{a_s}$ and $x^{(i_{s})}_{b_s}$
cannot be paired with $x^{(i_{s+1})}_{b_s}$
(as the pairing would be invalid),
each $a_s,b_s$ must be in an equivalence class of size at least $2$,
and thus $\beta\leq 2t/2=t$. Thus, we have that
$m^\alpha d^\beta \leq m^2 d^2 = m^{t/2+1} d^t$.

For our inductive step, suppose that we have $t>2$
and that our bound on $m^\alpha d^\beta$ holds for any valid pairing
with a smaller value of $t$. Given our valid pairing of size $t$,
we consider two cases based on whether or not any $i_s$ is
in an equivalence class all by itself.

If no $i_s$ is in its own equivalence class,
then there are at most $t/2$ equivalence classes
among the $i_s$'s and so (as $\beta$ is still at most $t$),
we have that $m^\alpha d^\beta \leq m^{t/2}d^t$.

Otherwise, suppose without loss of generality
that $i_t$ is in an equivalence class by itself.
If so, the four terms $x^{(i_{t})}_{a_t},  x^{(i_{t})}_{b_t}, x^{(i_{t})}_{a_{t-1}}, x^{(i_{t})}_{b_{t-1}}$
must be paired with each other. Since $a_t$ and $b_t$ cannot be equivalent,
$x^{(i_{t})}_{a_t}$ must pair with $x^{(i_{t})}_{a_{t-1}}$ or $ x^{(i_{t})}_{b_{t-1}}$.
Without loss of generality, we assume that $ x^{(i_{t})}_{a_t}$ is paired with
$x^{(i_{t})}_{a_{t-1}}$, and $x^{(i_{t})}_{b_t}$ is paired with  $x^{(i_{t})}_{b_{t-1}}$.

Here we further subdivide into subcases based on
whether $i_{t-1}$ is equivalent to $i_1$ under our pairing.
If they are not equivalent, we construct a new pairing of size $t-1$
by maintaining all of the pairs that do not contain an $i_t, a_t,$ or $b_t$ term
and pairing $x^{(i_{1})}_{a_{t-1}}$ and $x^{(i_{1})}_{b_{t-1}}$
with whatever $x^{(i_{1})}_{a_t}$ and $x^{(i_{1})}_{b_t}$ were previously paired with.
We note that this preserves the equivalence relations on the $i$'s and $a/b$'s,
producing a new valid pairing with $\alpha-1$ equivalence classes on the $i$'s
and $\beta$ equivalence classes on the $a/b$'s.
By the inductive hypothesis, we have that
$m^{\alpha-1}d^\beta \leq m^{(t-1)/2+1}d^{t-1}$.
Thus, we obtain that
$m^\alpha d^\beta \leq m^{t/2+1}d^t (m/d^2)^{1/2} <  m^{t/2+1}d^t$.
This completes the inductive step in this case.

If our pairing makes $i_1$ equivalent to $i_{t-1}$,
the above construction does not work,
as the pairing produced is not valid \new{(since two adjacent $i$'s would be equivalent)}. In this case,
we instead construct a smaller valid pairing of size $t-2$.
This is done by removing any pairs involving $i_t, a_t,b_t, i_{t-1}, a_{t-1}, b_{t-1}$
and
\begin{enumerate}
\item Pairing $x^{(i_{1})}_{a_{t-2}}$ and  $x^{(i_{1})}_{b_{t-2}}$
with whatever was previously paired with $x^{(i_{t-1})}_{a_{t-2}}$ and $x^{(i_{t-1})}_{b_{t-2}}$, respectively
\item Pairing whatever was previously paired with $x^{(i_{t-1})}_{a_{t-1}}$ with
whatever was previously paired with $x^{(i_{1})}_{a_t}$
\item Pairing whatever was previously paired with $x^{(i_{t-1})}_{b_{t-1}}$
with whatever was previously paired with $x^{(i_{1})}_{b_t}$.
\end{enumerate}
It is easy to check that this produces a pairing on all $2(t-2)$ terms
that does not introduce any equivalences that were not in the original pairing,
while losing the equivalence class of $i_t$, and possibly the equivalence classes of $a_t$ and $b_t$.
Thus, this new valid pairing has $\alpha-1$ equivalence classes
on the $i$'s and at most $\beta-2$ equivalence classes on the $a/b$'s.
Thus, by our inductive hypotheses, it follows that
$m^{\alpha-1} d^{\beta-2} \leq m^{(t-2)/2+1}d^{t-2}$,
and thus that $m^\alpha d^\beta \leq m^{t/2+1}d^t$. This completes our inductive step,
establishing Claim~\ref{clm:ind}.
\end{proof}

With this bound on $m^\alpha d^\beta$, we have that
$$
\E[\tr((A')^t)] \leq d^{-2t} \sum_{\textrm{valid pairings}} m^{t/2+1}d^t \;.
$$
The number of valid pairings is at most the number of pairings on a set of size $4t$, or $(4t-1)!! \leq (4t)^{2t}$.
Therefore, we have that
$$
\E[\tr((A')^t)] \leq m (m (4t)^4 / d^2)^{t/2} \;.
$$
\new{If $t$ is the nearest even integer to $\log(d)$,
we note that $m$ is still at most a small constant times $(d^2/(4t)^4)$. Thus, we have that}
$$
\E[\|A'\|_2^t] \leq \E[\tr((A')^t)] \leq (1/4)^t/d \;.
$$
Therefore, by the Markov bound, with probability at least $1-1/d$
we have that $\|A'\|_2 \leq 1/4.$

This complete the proof of Lemma~\ref{lem:A-prime}.
\end{proof}

Combining  the above bounds on $\|A^\ast\|_2$ and $\|A'\|_2$,
we have that with probability $1-o_d(1)$ that $\|A\|_2 \leq 1/2$,
proving Proposition \ref{A norm prop}.

\subsection{Some High Probability Bounds on $M, \eps,$ and $\delta$}\label{bounds sec}

In this section, we prove some facts about $M,\eps$ and $\delta$
and introduce some high-probability conditions on them. \new{In particular,
we will need to prove operator norm bounds on $M^{-1}$ and $\ell^\infty$ bounds on $\delta$ and $\eps$.}

To begin with, by Proposition \ref{A norm prop}, we have that
with high probability $\|A\|_2 <1/2.$ Call this event $E_1$.
Since $M \succeq (1-1/d)I - A$, this implies $M\succeq I/3.$
Note that $E_1$ implies that $M$ is invertible \new{and in particular that $\|M^{-1}\|_2 \leq 3$}.

We would next like to get a better understanding of $\eps_i$. We note that
$$
\eps_i = 1/\|x^{(i)}\|_2^2 - 1 = \frac{1}{1+(\|x^{(i)}\|_2^2-1)} - 1 = -(\|x^{(i)}\|_2^2-1) + (\|x^{(i)}\|_2^2-1)^2 + O(\|x^{(i)}\|_2^2-1)^3 \;.
$$
We note that $(\|x^{(i)}\|_2^2-1)$ is $1/d$ times
a sum of $d$ independent copies of $G^2-1$, where $G \sim N(0, 1)$.
Thus, we have that with high probability $|\|x^{(i)}\|_2^2-1| < \log(d)/\sqrt{d}$ for all $1\leq i\leq m$. Call this event $E_2$.
Furthermore, conditioned on $E_2$, $(\|x^{(i)}\|_2^2-1)$ has \new{superpolynomially} small mean
and has variance $O(1/d)$. Thus, conditioned on $E_2$, we have that
$|\eps_i| = O(\log(d)/\sqrt{d})$ for all $i$, $\E[\eps_i] = O(1/d)$, and $\E[\eps_i^2]=O(1/d)$.

Finally, we note that $\delta_i$ is the inner product of the $i^{\mathrm{th}}$ row of $M^{-1}$ with $\eps$.
Let $L$ be this $i^{\mathrm{th}}$ row. Note that $L$ and $\eps$ are independent.
Conditioned on $E_1$, we have that $\|L\|_2 \leq \|M^{-1}\|_2 \leq 3.$
Note that for $L$ fixed, conditioned on $E_2$, we have that
$$
\delta_i = \sum_{j=1}^m L_i \eps_i \;.
$$
By the Hoeffding Inequality, we have that
\begin{align*}
\Pr\left(\left|\delta_i - \sum_{i=1}^m L_i \E[\eps_i] \right| > t\right)  & < 2 \exp\left(-\frac{2t^2}{4\sum_{j=1}^m L_j^2 \|\eps_j \|_\infty^2 } \right)\\
& \leq 2 \exp\left(\frac{-\Omega(t^2)}{\|L\|_2^2 \log^2(d)/d} \right) \;.
\end{align*}
Taking $t = \log^2(d)/\sqrt{d}$, \new{we have for $d$ sufficiently large that}
with high probability
$\left|\delta_i - \sum_{i=1}^m L_i \E[\eps_i] \right| < \log^2(d)/\sqrt{d}$ for all $1\leq i\leq m$.

Furthermore, we can write
$$
\sum_{i=1}^m L_i \E[\eps_i] = O(1/d) \mathbf{1} \cdot L = O(1/d) \mathbf{1}^{\top}M^{-1}e_i = O(1/d) e_i^{\top} M^{-1}\mathbf{1} \;.
$$
Now, \new{conditioned on} $E_1$, we have that
$M=\mathbf{1}\mathbf{1}^{\top}/d + O(I_{d})$.
Thus, the vector $M ((d/m)\mathbf{1})$ is $\mathbf{1}$ plus a vector of $\ell_2$-norm at most $O(d/\sqrt{m})$.
Consequently, $M^{-1}\mathbf{1}$ is $(d/m)\mathbf{1}$ plus $M^{-1}$ times a vector of $\ell_2$-norm $O(d/\sqrt{m})$.
Since $\|M^{-1}\|_2 = O(1)$,
\new{we have that} $\|M^{-1}\mathbf{1}-(d/m)\mathbf{1}\|_2 = O(d/\sqrt{m})$.
Hence, we get that
$$
\sum_{i=1}^m L_i \E[\eps_i] = O(1/\sqrt{m}) \;.
$$
Therefore, with high probability,
we have that $|\delta_i| < \log^2(d)/\sqrt{d}$ for all $1\leq i \leq m$.
We call this event $E_3$.

Let $E$ be the event that events $E_1,E_2,E_3$ all hold.
We note that this happens with probability $1-o_d(1)$.
We next need to show that for any $u\in\mathcal{C}$ that, conditioned on $E$,
it holds that $\left| \sum_{i=1}^m \delta_i (u\cdot v_i)^2 \right| \leq 1/2$
with probability at least $1-2^{-Cd}$, for some sufficiently large universal constant $C$.

\subsection{Cover Argument: Completing the Proof}\label{cover sec}

Here we will show that for any unit vector $u$ that conditioned on $E$ with probability at least $1-2^{-Cd}$ over the choice of $v^{(i)},\eps_i$ that $\left| \sum_{i=1}^m \delta_i (u\cdot v^{(i)})^2 \right| \leq 1/2$. Applying a union bound over all $u\in \mathcal{C}$ will complete our proof of Theorem \ref{thm:main}.

Let $\beta$ be the vector with entries $\beta_i = (u\cdot v^{(i)})^2$ for $1\leq i\leq m$. We would like to prove a high probability bound on the $\ell_2$ norm of $\beta$. Unfortunately, there is a decently large (i.e. only exponentially small in $d$) chance that individual entries of $\beta$ will have size $\Omega(1)$, so we will have to deal with these entries separately. In particular, note that $|\beta_i| < d^{-1/4}$ except with probability $\exp(-\Omega(d^{3/4}))$ \new{by Lemma \ref{sphere facts lemma}}. As these events are independent for each $i$, except for with exponentially small probability we have that $|\beta_i| > d^{-1/4}$ for at most $O(d^{1/4})$ many different values of $i$. Call these the \emph{heavy coordinates}.

Next, we attempt to show a high probability bound for the $\ell_2$ norm of $\beta$ over the light coordinates. In particular, we do this by bounding
$$
\E[\exp(\alpha\|\beta_{\textrm{light}}\|_2^2)]
$$
where $\alpha = d^{9/8}$.
This is at most
$$
\prod_{i=1}^m \E[\exp(\min(\alpha\beta_i^2,\alpha/d^{1/2}))].
$$
To evaluate the expectation, we note that
$$
\E[\exp(\min(d\beta_i^2,d^{1/2}))] = 1+\int_1^{e^{\alpha/\sqrt{d}}} \Pr(\exp(\alpha\beta_i^2) > t) dt.
$$
The probability in question comes down to the probability that $\beta_i > \sqrt{\log(t)/\alpha}$. However, \new{Lemma \ref{sphere facts lemma}} implies that $\Pr(\beta_i > x) = \exp(-\Omega(dx))$. Thus, we have that
$$
\E[\exp(\min(\alpha\beta_i^2,\alpha/d^{1/2}))] = 1+\int_1^{e^{\alpha/\sqrt{d}}} \exp(-\Omega(d\sqrt{\log(t)/\alpha})) dt.
$$
We evaluate this integral separately on the ranges $[1,2]$ and $[2,e^{\sqrt{d}}].$ For $t\in [1,2]$ we have that $\log(t)=\Theta(t-1)$ and thus, letting $x=t-1$, the integral is
$$
\int_0^1 \exp(-\Omega(d\sqrt{x/\alpha}))dx = \alpha/d^2 \int_0^{d^2/\alpha} \exp(-\Omega(\sqrt{y}))dy = O(\alpha/d^2).
$$
To evaluate the integral between $2$ and $e^{\alpha/\sqrt d}$, we let $s=\log(t)$ and get
$$
\int_{\log(2)}^{\alpha/\sqrt d} \exp(-\Omega(d\sqrt{s/\alpha})+s)ds.
$$
The integrand here is at most $1/d^2$, and so the integral is $O(\alpha/d^2)$. Thus, we have that
$$
\E[\exp(\alpha\|\beta_{\textrm{light}}\|_2^2)] = (1+O(\alpha/d^2))^m = \exp(O(\alpha m/d^2)).
$$
Thus, by the Markov bound, except with probability $\exp(-\Omega(m/d^{7/8}))$, we have that $\alpha\|\beta_{\textrm{light}}\|_2^2 = O(\alpha m/d^2)$ or $\|\beta_{\textrm{light}}\|_2^2 = O(m/d^2)$.

So with exponentially high probability we can write $\beta = \beta_{\textrm{heavy}}+\beta_{\textrm{light}}$, 
where $\beta_{\textrm{heavy}}$ has support size $O(d^{1/4})$ 
and entries at most $1$ and $\beta_{\textrm{light}}$ 
has squared $\ell_2$-norm at most $O(m/d^2)$. 
We wish to bound $|\delta\cdot \beta|$. 
We begin by noting that
$$
|\delta\cdot \beta_{\textrm{heavy}}| \leq \|\delta\|_\infty \|\beta_{\textrm{heavy}}\|_1 \;.
$$
Event $E$ implies that $\|\delta\|_\infty \leq \log^2(d)/\sqrt{d}$, 
and as shown above, with high probability $\|\beta_{\textrm{heavy}}\|_1 = O(d^{1/4})$. 
Thus, conditioned on $E$, with exponentially high probability, 
we have that $|\delta\cdot \beta_{\textrm{heavy}}| = O(\log^2(d)/d^{1/4})$.

Next we bound $|\delta\cdot \beta_{\textrm{light}}|$. 
This can be rewritten as follows:
$$
\delta\cdot \beta_{\textrm{light}} = (M^{-1}\eps)\cdot \beta_{\textrm{light}} 
= \eps \cdot (M^{-1} \beta_{\textrm{light}}) \;.
$$
Condition $E_1$ implies that $M^{-1}$ has constant operator norm 
and we know that with high probability it holds that 
$\|\beta_{\textrm{light}}\|_2^2 = O(m/d^2)$. Thus, conditioned on $E_1$, 
with exponentially high probability we have that 
$\gamma:= M^{-1} \beta_{\textrm{light}}$ satisfies $\|\gamma\|_2^2 = O(m/d^2)$.

Note though that $\gamma$ depends only on $u$ and the $v_i$, 
and is thus independent of $\eps$ (even conditioned on $E_1$). 
We thus have that \new{conditioned on $E_1$ and $E_2$} 
that $\gamma \cdot \eps = \sum_{i=1}^m \gamma_i \eps_i$ 
is a sum of independent random variables. 
By an application of Hoeffding's Inequality, we obtain:
\begin{align*}
\Pr\left(\left|\gamma\cdot \eps - \sum_{i=1}^m \gamma_i \E[\eps_i] \right| > 1/6 \right)  & < 2 \exp\left(-\frac{1/18}{\sum_{j=1}^m \gamma_j^2 (2\|\eps_j\|_\infty)^2} \right)\\
& \leq 2 \exp\left(\frac{-\Omega(1)}{m\log^2(d)/d^3} \right) \;.
\end{align*}
Thus, with \new{exponentially high probability}, we have that
$$
\left|\gamma_{\textrm{light}}\cdot \eps - \sum_{i=1}^m \gamma_i \E[\eps_i] \right| < 1/6 \;.
$$
This leaves us to bound
\begin{align*}
\left| \sum_{i=1}^m \gamma_i \E[\eps_i]\right| & = O(1/d)\gamma\cdot \mathbf{1}\\
& \leq O(1/d) \|\gamma\|_2 \|\mathbf{1}\|_2\\
& = O(1/d)O(\sqrt{m}/d)\sqrt{m}\\
& = O(m/d^2) < 1/6 \;.
\end{align*}
Combining with the above, we find that conditioned on $E$
with exponentially high probability over the choice of $v^{(i)}$ 
and $\eps_i$ that $|\delta\cdot \beta|<1/2$.

This completes the proof of Theorem~\ref{thm:main}. \qed

\paragraph{Acknowledgements}
We thank the authors of~\cite{PTVW22} for pointing out a normalization
issue in our proof, and for encouraging us to publish our findings.

\bibliographystyle{alpha}
\bibliography{allrefs}

\newpage



\end{document}